\documentclass[11pt,leqno]{article}
\usepackage{amssymb, amscd, amsmath, amsthm}

\allowdisplaybreaks

\def\ve{\varepsilon}

\def\mod{\,\text{\rm mod}\;}
\def\beq{\begin{equation}}
\def\eeq{\end{equation}}
\def\cite#1{{\rm [#1]}}

\newtheorem{theorem}{Theorem}
\newtheorem{lemma}{Lemma}

\theoremstyle{definition}

\newtheorem{problem}{Problem}
\newtheorem*{rema}{Remark}

\begin{document}

\numberwithin{equation}{section}
\title{Are there arbitrarily long arithmetic progressions in the sequence of twin primes? II}

\author{J\'anos Pintz\thanks{Supported by OTKA Grants K72731, K67676 and ERC-AdG.228005.}}

\date{}
\maketitle

\section{Introduction}
\label{sec:1}

Until very recently the twin prime conjecture seemed to be completely inaccessible with available methods of number theory.
Four years ago, in a joint work with D. Goldston and C. Y{\i}ld{\i}r{\i}m \cite{GPY} we proved that assuming a very regular distribution of primes in arithmetic progressions we obtain a somewhat weaker result
\beq
\liminf_{n \to \infty} (p_{n + 1} - p_n) \leq 16,
\label{eq:1.1}
\eeq
where $p_n$ denotes the $n$\textsuperscript{th} prime.

The condition was that the level $\vartheta$ of distribution of primes, that is, an exponent, such that for $\ve > 0$, $A > 0$
\beq
\sum_{q \leq N^{\vartheta - \ve}} \max_{\substack{ a\\ (a,q) = 1}} \biggl| \sum_{\substack{p \equiv a \pmod q\\
p \leq N}} \log p - \frac{N}{\varphi(q)} \biggr| \ll_{\ve, A} \frac{N}{(\log N)^A}
\label{eq:1.2}
\eeq
satisfies $\vartheta \geq  0.971$; an assumption, just slightly weaker than the strongest possible hypothesis $\vartheta = 1$, the well-known Elliott--Halberstam \cite{EH} conjecture.
The best known admissible value for~$\vartheta$, the relation $\vartheta = 1/2$, is the celebrated Bombieri--Vinogradov theorem.
In the same work we proved that any $\vartheta > 1/2$ would yield infinitely many bounded gaps between primes, that is
\beq
\liminf_{n \to \infty} (p_{n + 1} - p_n) \leq C(\vartheta).
\label{eq:masodik1.2}
\eeq

Since we can not suppose concerning the level of distribution anything beyond the Elliott--Halberstam conjecture (which would yield also \eqref{eq:1.1}) the question arises, whether we can deduce the twin prime conjecture itself -- or perhaps even a positive answer for the question in the title of our paper -- from a hypothetical very regular behaviour of some related sequences (possibly including the primes themselves) similarly to the Elliott--Halberstam conjecture for primes.

We will give under some plausible hypotheses an affirmative answer for this question, including the existence of arbitrarily long arithmetic progressions in the sequence of twin primes.
Surprisingly the required distribution level is just $\vartheta > \frac34$ (or with more precise arguments even slightly less, $\vartheta \geq 0.7284$), but the sequences for which a suitable analogue of \eqref{eq:masodik1.2} is needed, are not just the primes but all the following ones:
\beq
\log p, \ \lambda(n), \ \lambda(n)\lambda(n + 2),\ \lambda(p + 2) \log p,\ \lambda(p - 2)\log p,
\label{eq:1.3}
\eeq
where $p$ denotes always primes, $\mathcal P$ the set of all primes.

We have to note that while for $\lambda(n)$ we know the analogue of the Bombieri--Vinogradov theorem
\beq
\sum_{q \leq N^{\vartheta - \ve}} \max_a \biggl| \sum_{\substack{n \equiv a\pmod q\\
n \leq N}} \lambda(n) \biggr| \ll_{\ve, A} \frac{N}{\log^A N}
\label{eq:1.4}
\eeq
with $\vartheta = 1/2$ (the proof of Vaughan \cite{Vau}, Theorem 4, with $\mu(n)$ in place of $\lambda(n)$ can be easily modified to yield \eqref{eq:1.4}), our knowledge about the other sequences is much more limited, since the following problems are still open (see \cite{Cho}, \cite{Hil}, \cite{Iwa}).

\begin{problem}
\label{pr:1}
Is $\sum\limits_{n \leq x} \lambda(n) \lambda(n + 2) = o(x)$ (see \cite{Cho, (341), p.~96; the quantity $O(X)$ there is a misprint, it has to be replaced by $o(x)$}), or even whether we have an absolute constant $c$ such that
\beq
\biggl| \sum_{n \leq x} \lambda(n) \lambda(n + 2) \biggr| < (1 - c) x \ \text{ for } \ x > x_0.
\label{eq:1.5}
\eeq
\end{problem}

\begin{problem}
\label{pr:2}
Are there infinitely many primes with $\lambda(p + 2) = -1$ (or $\lambda(p - 2) = -1$)?
\end{problem}

\noindent
{\bf Problem 2'.}
Are there infinitely many primes with $\lambda(p + 2) = 1$ (or $\lambda(p - 2) = 1$)?

\medskip
It may be worth to mention that the author succeeded to show very recently \cite{Pin2} the existence of a positive even $d \leq 18$ such that $\lambda(p + d) = -1$ for infinitely many primes~$p$.

We can more generally work with any fixed positive even integer $h$ in place of~$2$, so the same argument works for the generalized twin prime conjecture too.

\begin{theorem}
\label{th:1}
Suppose that with a $\vartheta = \vartheta_1 > 3/4$, the relations \eqref{eq:1.2}, \eqref{eq:1.4}, further the analogues of \eqref{eq:1.4} with $\lambda(n)$ replaced by $\lambda(n) \lambda(n + h)$, $\lambda (p - h)\log p$ and $\lambda(p + h)\log p$ hold, where $h$ is any positive even integer.
Then $p + h$ is prime for infinitely many primes~$p$.
\end{theorem}

The result can be proved with somewhat more effort under a slightly weaker condition for the distribution of the above sequences $(\vartheta \geq 0.7284)$.
Further we can give a lower estimate for the number of generalized twin primes up to~$N$ which is just a constant factor weaker than the expected number
\beq
\mathfrak S_0(h) \frac{N}{\log^2 N}, \quad \mathfrak S_0(h) := \prod\limits_{p \mid h} \left(1 - \frac1p\right)^{-1} \prod_{p \nmid h} \left( 1 - \frac1{(p - 1)^2} \right).
\label{eq:1.6}
\eeq

\begin{theorem}
\label{th:2}
Suppose that the conditions of Theorem~\ref{th:1} are valid for a $\vartheta \geq 0.7284$.
Then with an absolute constant~$c$ we have for any even $h > 0$
\beq
\#\{p \leq N; \ p, p + h \in \mathcal P\} \geq \frac{c \mathfrak S_0(h)N}{\log^2 N} \ \text{ for } N > N_1.
\label{eq:1.7}
\eeq
\end{theorem}

The estimate \eqref{eq:1.7} implies that as shown in \cite{Zho}, or in greater generality in \cite{Pin1}, the method of proof of Green--Tao \cite{GT} can be adapted to this situation, yielding

\begin{theorem}
\label{th:3}
Suppose the conditions of Theorem~\ref{th:1} for a $\vartheta_1 \geq 0.7231$, that is, that all $5$ functions in \eqref{eq:1.3} have distribution level~$\vartheta_1$, with $2$ replaced by $h$.
Then for any even $h > 0$ there are arbitrarily long arithmetic progressions such that $p + h$ is also prime for all elements of the progression.
\end{theorem}

\section{Proof of Theorem~\ref{th:1}}
\label{sec:2}

In the work \cite{GPY} we introduced for $k$-element sets $\mathcal H = \{h_i\}^k_{i = 1}$ the function
\beq
\Lambda_R(n; \mathcal H, l) := \frac1{(k + l)!} \sum_{\substack{d \leq R\\
d \mid P_{\mathcal H}(n)}} \mu(d) \left(\log \frac{R}{d}\right)^{k + l}, \quad P_{\mathcal H}(n) = \prod^k_{i = 1} (n + h_i),
\label{eq:2.1}
\eeq
which we will use now in the special case $k = 2$, $\mathcal H = \{0, h\}$.
However, instead of $a_n = \Lambda^2_R(n)$ as in \cite{GPY} we will weight now the integers with
\beq
b_n = a_n(1 - \lambda(n)) (1 - \lambda(n + h)) \geq 0, \quad a_n = \Lambda_R(n; \mathcal H, l)^2.
\label{eq:2.2}
\eeq
The singular series
\beq
\mathfrak S(\mathcal H) = \prod_p \left( 1 - \frac{\nu_p(\mathcal H)}{p} \right) \left(1 - \frac1p\right)^{-k},
\label{eq:2.3}
\eeq
where $\nu_p(\mathcal H)$ denotes the number of residue classes occupied by $\mathcal H \mod p$, reduces now to $\mathfrak S_0(h)$, given in \eqref{eq:1.6}.
The $k$-tuple $\mathcal H$ is called in general admissible if $\nu_p = \nu_p(\mathcal H) < p$ for all primes $p$, equivalently, if $\mathfrak S(\mathcal H) \neq 0$.
We remark that for any admissible $k$-tuple $\mathcal H = \mathcal H_k$, hence also in our case $\mathcal H = \{0, h\}$ we have
\beq
\mathfrak S(\mathcal H_k) \geq \prod_{p \leq 2k} \frac1p \prod_{p > k} \left(1 - \frac{k}{p}\right) \left(1 - \frac1{p}\right)^{-k} > c_0(k).
\label{eq:2.4}
\eeq
We quote from \cite{GPY} as our first two lemmas Propositions~1 and 2 (see \eqref{eq:2.14}--\eqref{eq:2.15}), which will form the base of our argument.
We will restrict ourselves for the case $\mathcal H = \mathcal H_1 = \mathcal H_2$, but keep the parameter~$l$, which will be used in Section~\ref{sec:3} to show the stronger Theorem~\ref{th:2}.
We will use the notation $\theta(n) = \log p$ if $n = p \in \mathcal P$, $\theta(n) = 0$ otherwise, $n \sim N$ for $n \in [N, 2N)$, $C$ an absolute constant whose value may be different at different occurrences.

\begin{lemma}
\label{lem:1}
If $R \ll N^{1/2}(\log N)^{-C}$ then
\beq
\frac1N \sum_{n \sim N} \Lambda_R(n; \mathcal H, l_1) \Lambda_R(n; \mathcal H, l_2) = (\mathfrak S(\mathcal H) + o(1)) {l_1 + l_2\choose l_1} \frac{(\log R)^{k + l_1 + l_2}}{(k + l_1 + l_2)!}.
\label{eq:2.5}
\eeq
\end{lemma}

\begin{lemma}
\label{lem:2}
If
 $R \ll N^{(\vartheta - \ve)/2}$ then for any $h \in \mathcal H$ we have
\begin{align}
&\frac1N \sum_{n \sim N} \Lambda_R(n; \mathcal H, l_1) \Lambda_R(n; \mathcal H, l_2)\theta(n + h) =\\
 & =(\mathfrak S(\mathcal H) + o(1)) {l_1 + l_2 + 2\choose l_1 + 1} \frac{(\log R)^{k + l_1 + l_2 + 1}}{(k + l_1 + l_2 + 1)!}.\notag
\label{eq:2.6}
\end{align}
\end{lemma}

We will need an analogous lemma for the sequences
\beq
f(n) = \lambda(n), \ \ \lambda(n) \lambda(n + h), \ \ \theta(n)\lambda(n + h), \ \ \lambda(n) \theta(n + h),
\label{eq:2.7}
\eeq
where we use the hypothesis that $f(n)$ satisfies the analogue of \eqref{eq:1.4}, that is,
\beq
\sum_{q \leq N^{\vartheta - \ve}} \max_a \biggl| \sum_{\substack{n \equiv a \pmod q\\
n \leq N}} f(n) \biggr| \ll_{\ve, A} \frac{N}{\log^A N}.
\label{eq:2.8}
\eeq

\begin{lemma}
\label{lem:3}
Suppose \eqref{eq:2.8} and $f(n) \ll (\log N)^C$.
If $A > 0$ arbitrary, $R \ll N^{(\vartheta - \ve)/2}$, then we have for any $\mathcal H = \{h_i\}^k_{i = 1}$
\beq
S_f(N) = \frac1N \sum_{n \sim N} \Lambda_R(n; \mathcal H, l_1) \Lambda_R(n; \mathcal H, l_2) f(n) \ll \frac{N}{\log^A N},
\label{eq:2.9}
\eeq
where the constant implied by the $\ll$ symbol depends on $k, l_i, C, A, \ve$.
\end{lemma}

\begin{proof}
For any squarefree $m$ and $\mathcal H = \{h_i\}^k_{i = 1}$ the number $\nu_m = \nu_m(\mathcal H)$ of the solution of the congruence
\beq
\prod^k_{i = 1} (n + h_i) \equiv 0 \pmod m
\label{eq:2.10}
\eeq
satisfies by the Chinese remainder theorem
\beq
\nu_m = \prod_{p\mid m} \nu_p \leq k^{\omega(m)} = d_k(m),
\label{eq:2.11}
\eeq
where $\omega(m)$ denotes the number of prime factors of~$m$, $d_k(m)$ the number of ways to write $m$ as a product of~$k$ integers.
Interchanging the order of summation we can write $S_f(N)$ with the notation $K = 2k + l_1 + l_2$ as
\begin{align}
&\frac1N \sum_{d \leq R} \sum_{e \leq R}
\frac{\mu(d) \mu(e) \left(\log \frac{R}{d}\right)^{k + l_1} \left(\log \frac{R}{d} \right)^{k + l_2}}{(k + l_1)! (k + l_2)!}
\sum_{\substack{n \sim N\\
[d, e] \mid P_{\mathcal H}(n)}} f(n) \\
&\ll \frac{\log^K R}{N} \sum_{q \leq R^2} \biggl(\sum_{q = [d, e]} 1\biggr) \nu_q E_N(q), \notag
\label{eq:2.12}
\end{align}
where (for $q \leq N$)
\beq
E_N(q) := \max_a \biggl| \sum_{\substack{n \sim N\\
n \equiv a\pmod q}} f(n) \biggr| \ll \frac{N(\log N)^C}{q}.
\label{eq:2.13}
\eeq

Using our hypotheses we obtain as in (9.13) of \cite{GPY}
\begin{align}
\label{eq:2.14}
&S_f(N) \ll \frac{\log^K R}{N} \biggl( \sum_{q \leq R^2} \frac{d_{3k}(q)^2}{q} \sum_{q \leq R^2} q E^2_N(q) \biggr)^{1/2} \ll \\
&\ll \frac{\log^K R}{N} \left((\log N)^{9k^2} N(\log N)^{C } \frac{N}{\log^A N}\right)^{1/2} \ll (\log N)^{K + (9k^2 + C - A)/2}.
\notag
\end{align}
\end{proof}

Using the notation
\beq
B_0 := B_0(R, \mathcal H, k, l) = {2l \choose l}  \frac{(\log R)^{k + 2l}}{(k + 2l)!} \mathfrak S(\mathcal H),
\label{eq:2.15}
\eeq
we have by Lemmas~\ref{lem:1} and \ref{lem:3} in the special case $\mathcal H_1 = \mathcal H_2 = \{0, h\}$, $k = 2$, $l_1 = l_2 = l = 0$
\beq
B := \sum_{n \sim N} b_n \sim \sum_{n \sim N} a_n \sim B_0 := \frac{\mathfrak S_0(h) N \log^2 R}{2} .
\label{eq:2.16}
\eeq

On the other hand we obtain from Lemmas \ref{lem:2} and \ref{lem:3} with the same choice $\mathcal H = \{0, h\}$, $l_1 = l_2 = 0$
\begin{align}
\label{eq:2.17}
P^* :&= \sum_{n \sim N} b_n (\theta(n) + \theta(n + h)) =\\
&= 2 \sum_{n \sim N} a_n \bigl\{ (1 - \lambda(n + h)) \theta(n) + (1 - \lambda(n)) \theta(n + h) \bigr\} \sim 4 \cdot 2 B_0 \cdot \frac{\log R}{3}. \notag
\end{align}

In order to have at least one prime pair $p, p + h$ with $p \in [N, 2N)$ we need to show with $R = N^{(\vartheta - \ve)/2}$
\beq
P^* - B \log(3N) > 0,
\label{eq:2.18}
\eeq
which is really true if
\beq
\frac83 \frac{\vartheta - \ve}{2} > 1 + \ve .
\label{eq:2.19}
\eeq
This is trivially true for any fixed $\vartheta > 3/4$ if $\ve$ is sufficiently small and $N$ sufficiently large.
This proves Theorem~\ref{th:1}.

\section{Proof of Theorem~\ref{th:2}}
\label{sec:3}

The proof of Theorem~\ref{th:2} needs a relatively simple modification, which allows to weaken slightly the constraint $\vartheta > 3/4$.
This can be achieved -- similarly to Section~3 of \cite{GPY} -- by applying a linear combination of the weights $\Lambda_R(n; \mathcal H, l)$ with $l = 0$ and $l = 1$.
More precisely we define
\beq
a'_n := a'_n(\mathcal H; u) = \left(\Lambda_R(n; \mathcal H, 0) + \frac{u(k + 1)}{\log R} \Lambda_R(n; \mathcal H, 1) \right)^2,
\label{eq:3.1}
\eeq
where $u$ is a real parameter to be chosen optimally later.
In our case $k = 2$ we obtain with the notation $B_0$ in \eqref{eq:2.16} for $\mathcal H = \{0, h\}$ from Lemmas~\ref{lem:1} and \ref{lem:2} in this case with the analogue $b'_n = a'_n(1 - \lambda(n))(1 - \lambda(n + h))$:
\beq
B'(N, \mathcal H, u) := \sum_{n \sim N} b'_n \sim \sum_{n \sim N} a'_n \sim B_0 \left(1 + 2u + 2u^2 \cdot \frac34\right).
\label{eq:3.2}
\eeq

The analogue of the evaluation of \eqref{eq:2.17} is now
\begin{align}
\label{eq:3.3}
P' :&= \sum_{n \sim N} b'_n(\theta(n) + \theta(n + h)) \sim \\
&\sim 2 \sum_{n \sim N} a'_n(\theta(n) + \theta(n + h))\sim \notag\\
&\sim 4 B_0 \log R \left( \frac23 + \frac{6u}{4} + \frac{18u^2}{20} \right). \notag
\end{align}
This means that we have to assure
\beq
P' - B' \log(3N) > 0
\label{eq:3.4}
\eeq
which will hold if we can find a $u$ with
\beq
g_u(\vartheta) = \vartheta\left(\frac43 + 3u + \frac{9 u^2}{5}\right) - \biggl( 1 + 2u + \frac{3u^2}{2}\biggr) > 0
\label{eq:3.5}
\eeq
if we choose $\ve$ sufficiently small.
The optimal choice for $u$ is $u = u_0 = \bigl(\sqrt{34} - 2\bigr)/9$, which yields a fixed positive lower bound $c_0$ for $g(\vartheta) = g_{u_0}(\vartheta)$ if
\beq
\vartheta \geq \vartheta_1 = 0.7231.
\label{eq:3.6}
\eeq
This is enough to obtain a weighted estimate for the number of generalized twin primes in $[N, 2N)$
\beq
\frac1N \sum_{\substack{n \sim N\\
n, n + h \in \mathcal P}} a'_n \log (3N) \geq c_1 \mathfrak S_0(h) \log^3 R.
\label{eq:3.7}
\eeq
However, if $n$ and $n + h$ are both primes then for $\mathcal H = \{0, h\}$ clearly
\beq
\Lambda_R(n; \mathcal H, l) = \frac1{(2 + l)!} (\log R)^{k + l} = \frac1{(2 + l)!}(\log R)^{2 + l},
\label{eq:3.8}
\eeq
consequently
\beq
a_n(\mathcal H, u_0) = \left(\frac{1 + u_0}{2}\right)^2 \log^4 R,
\label{eq:3.9}
\eeq
which by \eqref{eq:3.6} and \eqref{eq:3.7} leads to the estimate
\beq
\#\bigl\{p \in [N, 2N), \ p, p + h \in \mathcal P\bigr\} \geq \frac{c_2 \mathfrak S_0(h)N}{\log R \log N} \geq \frac{c_3 \mathfrak S_0(h) N}{\log^2 N} .
\label{eq:3.10}
\eeq

\begin{rema}
If we are allowed to choose a bigger $\vartheta$, then the lower estimate \eqref{eq:3.10} will improve but we do not reach the expected number corresponding to $c_3 = 1$ even supposing $\vartheta = 1$, the Elliott--Halberstam conjecture.
\end{rema}

\noindent
{\small J\'anos {\sc Pintz}\\
R\'enyi Mathematical Institute of the Hungarian Academy
of Sciences\\
Budapest\\
Re\'altanoda u. 13--15\\
H-1053 Hungary\\
E-mail: pintz@renyi.hu}


\begin{thebibliography}{99}

\bibitem[Cho]{Cho}
S. Chowla,
{\it The Riemann Hypothesis and Hilbert's tenth problem},
Gordon and Breach, New York, 1965.

\bibitem[EH]{EH}
P. D. T. A. Elliott,  H. Halberstam,
\textit{A conjecture in prime number theory},
{Symposia Mathematica} {\bf 4} INDAM, Rome,
 59--72, Academic Press, London, 1968/69.

\bibitem[GPY1]{GPY1}
D. A. Goldston, J. Pintz, C. Y. Y{\i}ld{\i}r{\i}m,
Primes in tuples I,
\textit{Ann. of Math. (2)} {\bf 170}
(2009), no.~2,  819--862.

\bibitem[Hil]{Hil}
A. J. Hildebrand,
\textit{Erd\H os' problems on  consecutive integers,
Paul Erd\H os and his Mathematics I}, Bolyai Society Mathematical
Studies {\bf 11}, 305--317, Budapest, 2002.

\bibitem[Iwa]{Iwa}
H. Iwaniec,
Prime numbers and $L$-functions,
{\it International Congress of Mathematicians}, Vol. I,
279--306, Eur. Math. Soc., Z\"urich, 2007.

\bibitem[Pin1]{Pin1}
J. Pintz, Are there arbitrarily long arithmetic progressions
in the sequence of twin primes?, preprint, arXiv:1002.2899

\bibitem[Pin2]{Pin2}
J. Pintz,
An approximation to the twin prime conjecture and the parity
phenomenon, preprint.

\bibitem[Vau]{Vau}
R. C. Vaughan, An Elementary Method in Prime Number Theory,
  {\it Recent progress in analytic number theory}, Vol. 1 (Durham, 1979),
341--348, Academic Press, London--New York, 1981.

\end{thebibliography}
\end{document}